\documentclass[reqno]{amsart}

\usepackage[english, activeacute]{babel}
\usepackage{amsmath,amsthm,amsxtra}
\usepackage{fix-cm}
\usepackage{graphics}
\usepackage{fancyhdr}
\usepackage{float}
\usepackage{anysize}
\usepackage{amssymb,mathrsfs}
\usepackage{enumerate}
\usepackage{epsfig}
\usepackage{latexsym}
\usepackage{amsfonts}
\usepackage{hyperref}
\usepackage{dsfont}
\usepackage{color}
\usepackage{tikz,pgfplots,pgf}
\tikzset{node distance=2cm, auto}
\usetikzlibrary{calc}
\usetikzlibrary{trees}
\usepackage{epstopdf}
\usepackage{xfrac}
\usepackage{tikz}
\usetikzlibrary{automata,positioning,matrix}

\newtheorem{theorem}{Theorem}[section]
\newtheorem*{theorem*}{Theorem}
\newtheorem{corollary}[theorem]{Corollary}

\newtheorem{lemma}[theorem]{Lemma}
\newtheorem{proposition}[theorem]{Proposition}

\theoremstyle{definition}
\newtheorem{definition}[theorem]{Definition}
\newtheorem{remark}[theorem]{Remark}

\newcommand{\Sum}{\displaystyle\sum}
\newcommand{\Lim}{\displaystyle\lim}

\newcommand{\Frac}[2]{\displaystyle\frac{#1}{#2}}
\newcommand{\Prod}{\displaystyle\prod}

\title[The vortex-like behavior of the R.z.f. to the right of the critical strip]{The vortex-like behavior of the Riemann zeta function\\ to the right of the critical strip}
\bigskip

\author{ {\bf Juan Mat\'{i}as Sepulcre, Tom\'{a}s Vidal}}
\address{Department of Mathematics\\ University of
Alicante, 03080-Alicante\\
Spain}
\email{JM.Sepulcre@ua.es, tmvg@alu.ua.es}
\urladdr[First author]{\texttt{http://personal.ua.es/en/jm-sepulcre/}}
\bigskip

\begin{document}
 \begin{abstract}
Based on an equivalence relation that was established recently on exponential sums,
in this paper we study the class of functions that
are equivalent to the Riemann zeta function in the half-strip $\{s\in\mathbb{C}:\operatorname{Re}s>1\}$. In connection with this class of functions, we first determine the value of the maximum abscissa from which the images of any function in it cannot take  a prefixed argument.
The main result shows that each of these functions experiments a vortex-like
behavior in the sense that the main argument of its images varies indefinitely near the vertical line $\operatorname{Re}s=1$. In particular, regarding the Riemann zeta function $\zeta(s)$, for every $\sigma_0>1$ we can assure the existence of a relatively dense set of real numbers $\{t_m\}_{m\geq 1}$ such that the parametrized curve traced by the points  $(\operatorname{Re}(\zeta(\sigma+it_m)),\operatorname{Im}(\zeta(\sigma+it_m)))$, with $\sigma\in(1,\sigma_0)$, makes a prefixed finite number of turns around the origin.
\end{abstract}

\maketitle



\vspace{0.55cm}

\textit{Secondary} 11K60; 30D20; 30Axx.

\medskip



\section{Introduction}

General Dirichlet series consist of those exponential sums that take the form
$$\Sum_{n\geq 1}a_ne^{-\lambda_n s},\ a_n\in\mathbb{C},\ s=\sigma+it,$$
 where
$\{\lambda_n\}$ is a strictly increasing sequence of posi\-tive numbers
tending to infinity.
In particular, it is widely known that the Riemann zeta function $\zeta(s)$, which plays a pivotal role
in analytic number theory, is defined as the analytic continuation of the function defined for $\sigma>1$ by the sum
$\sum_{n=1}^{\infty}\frac{1}{n^s}$, which constitutes a classical Dirichlet series.

In the beginnings of the 20th century, H. Bohr gave important steps in the understanding
of Dirichlet series and their regions of convergence, uniform convergence and absolute
convergence. As a result of his investigations on these functions, Bohr introduced an equivalence relation among them that led to the so-called Bohr's equivalence theorem, which shows that equivalent Dirichlet series take the same values in certain vertical lines or strips in the complex plane (e.g. see \cite{Apostol,BohrDirichlet,Rigue,Spira}).

Precisely, by analogy with Bohr's theory, we established in \cite{SV} an equi\-valence relation $\sim$ on the classes $\mathcal{S}_{\Lambda}$ consisting of exponential sums of the form
 \begin{equation}\label{eqqnew}
\sum_{j\geq 1}a_je^{\lambda_jp},\ a_j\in\mathbb{C},\ \lambda_j\in\Lambda,
\end{equation}
where $\Lambda=\{\lambda_1,\lambda_2,\ldots,\lambda_j,\ldots\}$ is an arbitrary countable set of distinct real numbers, and $p$ is a parameter (in our case, it will be changed by $s =
 \sigma+ it$ in the complex case, or by $t$ in the real case). In the context of almost periodic functions, to which this equivalence relation can also be extended, the main result of \cite{SV} refined Bochner's result in the sense that it was proved that the condition of almost periodicity is equivalent to that fact that every sequence of (vertical) translates has a subsequence that converges uniformly to an equivalent function (see also \cite{SV2}).
Throughout this work, we will use a generalization of Bohr's equivalence relation, defined in Section \ref{sect2}, which was used in  \cite{new} to get a a result like Bohr's equivalence theorem extended to certain classes of almost periodic functions in vertical strips $\{s\in\mathbb{C}:\alpha<\operatorname{Re}s<\beta\}$.

Regarding the Riemann zeta function $\zeta(s)$, the connection between it and prime numbers was discovered by L. Euler, who proved the identity
\begin{equation*}
\zeta(s)=\Prod_{k=1}^{\infty}\Frac{1}{1-p_k^{-s}}\ \mbox{for
}\operatorname{Re}s>1,
\end{equation*}
where the product on the right hand extends over all prime numbers
$p_k$. In view of the Euler product, it is easily seen that
$\zeta(s)$ has no zeros in the half-plane $\sigma>1$.
It is also known that the Dirichlet series and the Euler product of $\zeta(s)$ converge
absolutely in the half-plane $\sigma>1$ and uniformly in
$\sigma\geq 1+\delta$ for any $\delta>0$. A more advanced introduction to the theory surrounding the Riemann hypothesis can be found for example in \cite{Bor}.

In general terms, let $\{a_2,a_3,a_5,\ldots,a_{p_j},\ldots\}$ be
an arbitrary sequence of complex numbers such that $|a_{p_j}|=1$
for each $j=1,2,\ldots$. Take $a_1=1$ and define
$a_n=a_{p_1}^{\alpha_1}a_{p_2}^{\alpha_2}\cdots
a_{p_{k_n}}^{\alpha_{k_n}}$ when
$n=p_1^{\alpha_1}p_2^{\alpha_2}\cdots p_{k_n}^{\alpha_{k_n}}$ is
not a prime number. Throughout this paper, associated with such a sequence $\{a_2,a_3,a_5,\ldots\}$, we will consider the generic exponential sum
\begin{equation}\label{unoii}
\Sum_{n=1}^{\infty}\Frac{a_n}{n^s}
\end{equation}
which also converges absolutely in $\sigma>1$ and uniformly in
$\sigma\geq 1+\delta$ for any $\delta>0$. For example, the choices $a_{p_j}=1$ for each $j=1,2,\ldots$ and $a_{p_j}=-1$ for each $j=1,2,\ldots$ provide respectively the Riemann zeta function and the Dirichlet series of the Liouville function $\lambda(n)=(-1)^{\Omega(n)}$, where $\Omega(n)$ is the number of prime factors of $n$ (counted with multiplicities). Precisely, we show in Proposition \ref{cormontel} that the image of these two functions on the real axis provides the above and below bounds for
the absolute value of the image of each exponential sum of type \eqref{unoii} throughout every vertical line or closed half-strip in
$\{s\in\mathbb{C}:\operatorname{Re}s>1\}$.

With respect to the arguments of all exponential sums of type \eqref{unoii}, the special choices $a_{p_j}=i$ for each $j=1,2,\ldots$ and $a_{p_j}=-i$ for each $j=1,2,\ldots$ provide bounds for the change of these arguments (see Lemma \ref{lemimco}) and, in fact, they allows us to determine the maximum abscissa from
which the images of any exponential sum of type \eqref{unoii} cannot take a prefixed argument (see Lemma \ref{lemau}). In particular, this yields that the images of any function equivalent to the Riemann zeta function cannot take negative real values on a certain half-plane of the form $\{s\in\mathbb{C}:\operatorname{Re}s>\sigma_{\pi}\}$, with $\sigma_{\pi}>1$.

Likewise,
from Euler-type
product formula for these sums $S(s)$ of type \eqref{unoii} (see Lemma \ref{lemcu}),
the main result of our paper shows that each one of these functions, and in particular the Riemann zeta function, experiments a vortex-like behavior in the sense that, given $\sigma_0>1$ and $n\in\mathbb{N}$, there exists a relatively dense set of real numbers $\{t_{n,m}\}_{m\geq 1}$ such that, for each $m=1,2,\ldots$, the image of the vector-valued function $(\operatorname{Re}S(\sigma+it_{n,m}), \operatorname{Im}S(\sigma+it_{n,m}))$, for $\sigma$ in the interval $(1,\sigma_0)$, traces a curve in the plane which makes at least $n$ turns around the origin (see Theorem \ref{thm2} in this paper, and related results in Lemma \ref{kema} and propositions \ref{thm} and \ref{cor8}). To the best of our knowledge, this result has not been reported in the literature.

\section{The class of functions equivalent to the Riemann zeta function}\label{sect2}

Based on the Bohr's equivalence relation, which was considered in \cite[p.173]{Apostol} for general Dirichlet series, we defined in \cite{SV,SV2,new} new equivalence relations in the more general context of the classes $S_\Lambda$ of exponential sums of type \eqref{eqqnew}. In this paper, we will use the following definition which constitutes the same equivalence relation as that of \cite[Definition 2]{new}.

\begin{definition}\label{DefEquiv00}
Given $\Lambda=\{\lambda_1,\lambda_2,\ldots,\lambda_j,\ldots\}$ a set of distinct real numbers, consider $A_1(p)$ and $A_2(p)$ two exponential sums in the class $\mathcal{S}_{\Lambda}$, say
$A_1(p)=\sum_{j\geq1}a_je^{\lambda_jp}$ and $A_2(p)=\sum_{j\geq1}b_je^{\lambda_jp}.$
We will say that $A_1(p)$ is equivalent to $A_2(p)$ if there exists a $\mathbb{Q}$-linear map $\psi:\operatorname{span}_{\mathbb{Q}}(\Lambda)\to\mathbb{R}$ such that
$b_j=a_je^{i\psi(\lambda_j)}$ for each $j=1,2,\ldots$. 
\end{definition}

Let
$G_{\Lambda}=\{g_1, g_2,\ldots, g_k,\ldots\}$ be a basis of the
vector space over the rational numbers generated by a set $\Lambda=\{\lambda_1,\lambda_2,\ldots,\lambda_j,\ldots\}$, 
which implies that $G_{\Lambda}$ is linearly independent over the rational numbers and each $\lambda_j$ is expressible as a finite linear combination of terms of $G_{\Lambda}$, say
\begin{equation*}\label{errej}
\lambda_j=\sum_{k=1}^{q_j}r_{j,k}g_k,\ \mbox{for some }r_{j,k}\in\mathbb{Q}.
\end{equation*}
By abuse of notation, we will say that $G_{\Lambda}$ is a basis for $\Lambda$. Moreover, we will say that $G_{\Lambda}$ is an integral basis for $\Lambda$ when $r_{j,k}\in\mathbb{Z}$ for each $j,k$, i.e. $\Lambda\subset \operatorname{span}_{\mathbb{Z}}(G_{\Lambda})$ (it is worth noting that all the results of \cite{SV}
which can be formulated in terms of an integral basis are also valid under Definition \ref{DefEquiv00}).

In the particular case of the Riemann zeta function $\zeta(s)=\sum_{n\geq 1}\frac{1}{n^s}$, with $\operatorname{Re}s>1$, we can take  $\{\log2,\log3,\ldots,\log p_k,\ldots\}$, where $p_k$ is the $k$-th prime number, as an integral basis for the set $\Lambda$.
Likewise, the set of exponential sums which are equivalent to $\zeta(s)$ are given by
the series, for every choice of $\mathbf{x}=(x_1,x_2,\ldots,x_k,\ldots)\in%
\mathbb{R}^{\infty}$, of the form
\begin{equation}\label{unoi}
\zeta_{\mathbf{x}}(s):=\sum_{n\geq1}e^{<\mathbf{r}_n,\mathbf{x}>i}n^{-s},\operatorname{ with } \operatorname{Re}s>1,
\end{equation}
where $\mathbf{r}_n$ is the vector of integer components satisfying $\log n=<\mathbf{r}_n,\mathbf{g}>$ with $\mathbf{g}=(\log2,\log3,\ldots,\log p_k,\ldots)$ (see, for instance, \cite[Section 8.8]{Apostol}, \cite[Proposition 1]{SV} or \cite[Expression (2.2)]{new}).
In particular, the vector $\mathbf{x}=(\pi,\pi,\ldots,\pi,\ldots)$ generates the Dirichlet series for the Liouville function, denoted here as $\zeta_{\boldsymbol{\pi}}(s)$, which is related to the Riemann zeta function \cite{Lehman,Titchmarsh} by
\begin{equation*}\label{zliouville}
\zeta_{\boldsymbol{\pi}}(s)=\sum_{n=1}^{\infty}\frac{\lambda(n)}{n^s}=\Frac{\zeta(2s)}{\zeta(s)},\ \mbox{for }\operatorname{Re}s>1,
\end{equation*}
where $\lambda(n)$ is the Liouville's function.
Hence $\zeta_{\boldsymbol{\pi}}(s)$ has no singular points in the domain $\operatorname{Re}s>1$.

Now, let $\{a_2,a_3,a_5,\ldots,a_{p_j},\ldots\}$ be
an arbitrary sequence of complex numbers such that $|a_{p_j}|=1$
for each $j=1,2,\ldots$. Take $a_1=1$ and define
$a_n=a_{p_1}^{\alpha_1}a_{p_2}^{\alpha_2}\cdots
a_{p_{k_n}}^{\alpha_{k_n}}$ when
$n=p_1^{\alpha_1}p_2^{\alpha_2}\cdots p_{k_n}^{\alpha_{k_n}}$ is
not a prime number. Note that the exponential sum $\sum_{n\geq 1}\frac{a_n}{n^s}$ is identified with that of \eqref{unoi} given by
$\zeta_{\mathbf{x}}(s)=\sum_{n\geq1}\frac{e^{<\mathbf{r}_n,\mathbf{x}>i}}{n^{s}}$, where $\mathbf{x}=(x_1,x_2,\ldots,x_k,\ldots)$ satisfies $a_{p_j}=e^{ix_j}$ for each $j=1,2,\ldots$. Indeed, we have that
$$a_n=a_{p_1}^{\alpha_1}a_{p_2}^{\alpha_2}\cdots
a_{p_{k_n}}^{\alpha_{k_n}}=e^{i\left(\alpha_1x_1+\alpha_2x_2+\ldots+\alpha_{k_n}x_{k_n}\right)}=e^{<\mathbf{r}_n,\mathbf{x}>i},$$
where $\mathbf{r}_n$ is defined above. Consequently, we will handle these series written in the form of \eqref{unoi}.

Although the following preliminary results are reasonably simple, we next provide their proof for the sake of completeness.
We first obtain an Euler-type
product for $\zeta_{\mathbf{x}}(s)$ on $\operatorname{Re}s>1$ and we
prove that the convergence of this Euler-type product is
uniform in every half-plane $\operatorname{Re}s\geq 1+\delta$,
$\delta>0$.

\begin{lemma}\label{lemcu}
Given $\mathbf{x}=(x_1,x_2,\ldots,x_k,\ldots)\in\mathbb{R}^{\infty}$, the product
$\Prod_{k=1}^{\infty}\Frac{1}{1-e^{ix_{k}}p_k^{-s}}$ converges
uniformly to $\zeta_{\mathbf{x}}(s)$ in every half-plane
$\operatorname{Re}s\geq 1+\delta$, with $\delta>0$.
\end{lemma}
\begin{proof}
Since
$\Frac{1}{2^s}\zeta_{\mathbf{x}}(s)=\Sum_{n=1}^{\infty}\Frac{e^{<\mathbf{r}_n,\mathbf{x}>i}}{(2n)^s}$
for any $s$ with $\operatorname{Re}s>1,$ the function
$\zeta_{\mathbf{x}}(s)$ verifies
$$\zeta_{\mathbf{x}}(s)(1-e^{ix_{1}}2^{-s})=1+\frac{e^{ix_{2}}}{3^s}+\frac{e^{ix_{3}}}{5^s}+\frac{e^{ix_{4}}}{7^s}+\frac{e^{2ix_{2}}}{9^s}+\frac{e^{ix_{5}}}{11^s}+\ldots$$
Analogously,
$$\zeta_{\mathbf{x}}(s)(1-e^{ix_{1}}2^{-s})(1-e^{ix_{2}}3^{-s})=1+\frac{e^{ix_{3}}}{5^s}+\frac{e^{ix_{4}}}{7^s}+\frac{e^{ix_{5}}}{11^s}+\ldots$$
and, in general, for the first $m$ primes we have
$$\zeta_{\mathbf{x}}(s)\prod_{k=1}^m(1-e^{ix_{k}}p_k^{-s})=1+\frac{e^{<\mathbf{r}_{l_1},\mathbf{x}>i}}{l_1^s}+\frac{e^{<\mathbf{r}_{l_2},\mathbf{x}>i}}{l_2^s}+\ldots,$$
where $l_1,l_2,\ldots$ are the natural numbers which are not divisible
by any of the $m$ first prime numbers, and hence
$\{l_1,l_2,\ldots\}\subset\{p_{m+1},p_{m+1}+1,\ldots\}$. In this
way, given $\delta>0$, for any $s=\sigma+it$ such that $\sigma\geq
1+\delta$ we get
$$\left|\zeta_{\mathbf{x}}(s)-\prod_{k=1}^m\Frac{1}{1-e^{ix_{k}}p_k^{-s}}\right|=
\left|\prod_{k=1}^m\Frac{1}{1-e^{ix_{k}}p_k^{-s}}\right|\left|\zeta_{\mathbf{x}}(s)\prod_{k=1}^m(1-e^{ix_{k}}p_k^{-s})-1\right|=$$
$$=\prod_{k=1}^m\Frac{1}{\left|1-e^{ix_{k}}p_k^{-s}\right|}\left|\frac{e^{<\mathbf{r}_{l_1},\mathbf{x}>i}}{l_1^s}+\frac{e^{<\mathbf{r}_{l_2},\mathbf{x}>i}}{l_2^s}+\ldots\right|\leq\prod_{k=1}^m\Frac{1}{1-p_k^{-\sigma}}\left(\frac{1}{l_1^{\sigma}}+\frac{1}{l_2^{\sigma}}+\ldots\right)\leq$$
$$\leq\zeta(1+\delta)\sum_{k=p_{m+1}}^{\infty}\frac{1}{k^{1+\delta}},$$
which tends to $0$ as $m\rightarrow\infty$. Hence the result
holds.
\end{proof}

The following result, which is a clear consequence of the Euler product representation, shows that the functions $\zeta_{\mathbf{x}}(s)$ are bounded throughout every vertical line or closed strip in
$\{s\in\mathbb{C}:\operatorname{Re}s>1\}$.

\begin{proposition}\label{cormontel}
Let $\mathbf{x}=(x_1,x_2,\ldots,x_k,\ldots)\in\mathbb{R}^{\infty}$ and $K$ be a closed strip in
$\{s\in\mathbb{C}:\operatorname{Re}s>1\}$. Then
$$\zeta _{\boldsymbol{\pi}}(\sigma_0)\leq |\zeta_{\mathbf{x}}(s)|\leq \zeta(\sigma_0)\ \mbox{for any
} s\in K,$$ where $\sigma_0=\min\{\operatorname{Re}s:s\in K\}$.
\end{proposition}
\begin{proof}
Let $K$ be a strip in
$\{s\in\mathbb{C}:\operatorname{Re}s>1\}$ with
$\sigma_0=\inf\{\operatorname{Re}s:s\in K\}>1$ and
take $s=\sigma+it\in K$.
By the Euler-type product formula, we have
\begin{equation*} |\zeta_{\mathbf{x}}(\sigma+it)|=\left|\prod_{k\geq1}\frac{1}{1-e^{ix_{k}}p_k^{-\sigma-it}}\right|\geq
\prod_{k\geq 1}\frac{1}{1+p_k^{-\sigma}}\geq
\prod_{k\geq 1}\frac{1}{1+p_k^{-\sigma_0}}=\zeta_{\boldsymbol{\pi}}(\sigma_0).
\end{equation*}
Moreover, it is clear that
$$|\zeta_{\mathbf{x}}(s)|=\left|\Sum_{n\geq 1}\Frac{e^{<\mathbf{r}_n,\mathbf{x}>i}}{n^{\sigma+it}}\right|\leq\Sum_{n\geq 1}\Frac{1}{n^{\sigma}}\leq \Sum_{n\geq 1}\Frac{1}{n^{\sigma_0}}=\zeta(\sigma_0).$$
Thus the result holds.
\end{proof}

\section{On the values of the arguments of the functions that are equivalent to the Riemann zeta function}\label{sv}

Let $\zeta_{\mathbf{x}}(s)=\Sum_{n\geq1}e^{<\mathbf{r}_j,\mathbf{x}>i}e^{-s\log n}$, with $\operatorname{Re}s>1$ and $\mathbf{x}\in \mathbb{R}^{\infty}$, be an exponential sum which is equivalent to the Riemann zeta function.
By Lemma \ref{lemcu} we know that $\zeta_{\mathbf{x}}(s)$ can be expressed in terms of the Euler-type product $\Prod_{k=1}^{\infty}\Frac{1}{1-e^{ix_{k}}p_k^{-s}}$ and this product converges uniformly to $\zeta_{\mathbf{x}}(s)$ in every reduced strip of $U=\{s\in\mathbb{C}:\operatorname{Re}s>1\}$. This implies that the principal value of the argument of $\zeta_{\mathbf{x}}(s)$ can be written in terms of
\begin{equation*}\label{champa}
\operatorname{Arg}(\zeta_{\mathbf{x}}(s))=-\sum_{k\geq 1}\operatorname{Arg}(1-e^{ix_{k}}p_k^{-s})\ (\mbox{mod} (-\pi,\pi]).
\end{equation*}
In our case, given $\mathbf{x}\in \mathbb{R}^{\infty}$, we will handle the mapping $A_{\zeta_{\mathbf{x}}}(s):U\mapsto \mathbb{R}_+\cup\{\infty\}$ defined as
\begin{equation}\label{champa2}
A_{\zeta_{\mathbf{x}}}(s):=\sum_{k\geq 1}\left|\operatorname{Arg}(1-e^{ix_{k}}p_k^{-s})\right|.
\end{equation}
Notice that $\lim_{\sigma\to\infty}\zeta(\sigma+it)=1$ for any $t\in\mathbb{R}$ (and hence $\lim_{\sigma\to\infty}\operatorname{Arg}(\zeta(\sigma+it))=0$ for any $t\in\mathbb{R}$). So, thanks to \cite[Theorem 18]{new}, we state that every function $\zeta_{\mathbf{x}}(s)$ satisfies $\lim_{\sigma\to\infty}\zeta_{\mathbf{x}}(\sigma+it)=1$ for any $t\in\mathbb{R}$ (and hence $\lim_{\sigma\to\infty}\operatorname{Arg}(\zeta_{\mathbf{x}}(\sigma+it))=0$ for any $t\in\mathbb{R}$). In particular, we have that $\lim_{\sigma\to\infty}A_{\zeta_{\mathbf{x}}}(\sigma+it)=0$ for any $t\in\mathbb{R}$.

The following lemma shows the importance of $\zeta_{\boldsymbol{\frac{ \pi}{ 2}}}(s)=\prod_{k=1}^{\infty}\frac{1}{1-ip_k^{-s}}$ and $\zeta_{\boldsymbol{-\frac{ \pi}{ 2}}}(s)=\prod_{k=1}^{\infty}\frac{1}{1+ip_k^{-s}}$ in terms of the mapping $A_{\zeta_{\mathbf{x}}}(s)$, which is connected with the argument of the functions $\zeta_{\mathbf{x}}(s)$ equivalent to the Riemann zeta function.

\begin{lemma}\label{lemimco}
Let $A_{\zeta_{\mathbf{x}}}(s)=\sum_{k\geq 1}\left|\operatorname{Arg}(1-e^{ix_{k}}p_k^{-s})\right|$ be the mapping defined above.
It is satisfied:
\begin{itemize}
\item[i)] $A_{\zeta_{\mathbf{x}}}(\sigma+it)\leq A_{\zeta_{\boldsymbol{\frac{ \pi}{ 2}}}}(\sigma)=A_{\zeta_{\boldsymbol{-\frac{ \pi}{ 2}}}}(\sigma)$
for any $\mathbf{x}\in \mathbb{R}^{\infty}$ and $\sigma+it\in U$.
\item[ii)] $A_{\zeta_{\boldsymbol{\frac{ \pi}{ 2}}}}(\sigma)=\Sum_{k\geq 1}\arctan(p_k^{-\sigma})$ for any $\sigma>1$;
\item[iii)] $A_{\zeta_{\boldsymbol{\frac{ \pi}{ 2}}}}(\sigma)<\infty$ for any $\sigma>1$, and $\Lim_{\sigma\to 1^+}A_{\zeta_{\boldsymbol{\frac{ \pi}{ 2}}}}(\sigma)=\infty$;
\item[iv)] The function $A_{\zeta_{\boldsymbol{\frac{ \pi}{ 2}}}}:(1,\infty)\mapsto \mathbb{R}_+$, defined as $$A_{\zeta_{\boldsymbol{\frac{ \pi}{ 2}}}}(\sigma)=\sum_{k\geq 1}\left|\operatorname{Arg}(1- ip_k^{-\sigma})\right|,$$ is continuous and decreasing.
\end{itemize}
\end{lemma}
\begin{proof}
i) Given $\sigma>1$ and $k\in\mathbb{N}$, it is clear that $\left|\operatorname{Arg}(1-e^{ix_{k}}p_k^{-\sigma})\right|$ attains the maximum value when $x_{k}=\pm\frac{ \pi}{ 2}$. In the same way, fixed $\sigma+it\in\mathbb{C}$ with $\sigma>1$, we have $\left|\operatorname{Arg}(1-e^{ix_{k}}p_k^{-\sigma-it})\right|=\left|\operatorname{Arg}(1-e^{i(x_{k}-t\log p_k)}p_k^{-\sigma})\right|\leq \left|\operatorname{Arg}(1+ ip_k^{-\sigma})\right|=\left|\operatorname{Arg}(1-ip_k^{-\sigma})\right|$ for each $k=1,2,\ldots$, which proves i).

ii) Given $\sigma>1$, note that
$$\operatorname{Arg}(1-ip_k^{-\sigma})=\arctan(-p_k^{-\sigma})=-\arctan(p_k^{-\sigma})$$
and
$$\operatorname{Arg}(1+ip_k^{-\sigma})=\arctan(p_k^{-\sigma}).$$
Hence
$$A_{\zeta_{\boldsymbol{\pm\frac{ \pi}{ 2}}}}(\sigma)=\sum_{k\geq 1}|\arctan(p_k^{-\sigma})|=\sum_{k\geq 1}\arctan(p_k^{-\sigma}).$$

iii) For $\sigma>0$, it is easy to prove that
$\Frac{3\sigma}{3+\sigma^2}<\arctan\sigma<\sigma$ (see for example \cite[p. 665]{Ineq} for the first inequality). 
Hence
$$\sum_{k\geq 1}\frac{3p_k^{-\sigma}}{3+p_k^{-2\sigma}}\leq \sum_{k\geq 1}\arctan(p_k^{-\sigma})\leq \sum_{k\geq 1}p_k^{-\sigma}.$$
Moreover, we have that $\Sum_{k\geq 1}\frac{ 1}{ p_k^{\sigma}}$ and $\Sum_{k\geq 1}\frac{3/p_k^{\sigma}}{3+1/p_k^{2\sigma}}$ have the same character of convergence.
Consequently,
$$\sum_{k\geq 1}\arctan(p_k^{-\sigma})=A_{\zeta_{\boldsymbol{\frac{ \pi}{ 2}}}}(\sigma)<\infty\mbox{ for any }\sigma>1$$
and
$$\lim_{\sigma\to 1^+}\sum_{k\geq 1}\arctan(p_k^{-\sigma})=\lim_{\sigma\to 1^+}A_{\zeta_{\boldsymbol{\frac{ \pi}{ 2}}}}(\sigma)=\infty.$$

iv) It is clear that $\lim_{\sigma\to\infty}A_{\zeta_{\boldsymbol{\frac{ \pi}{ 2}}}}(\sigma)=0$. Furthermore, if $1<\sigma_1<\sigma_2$ then
$$\left|\operatorname{Arg}(1-ip_k^{-\sigma_1})\right|>\left|\operatorname{Arg}(1-ip_k^{-\sigma_2})\right|\mbox{ for each }k=1,2,\ldots$$
Consequently,
$$A_{\zeta_{\boldsymbol{\frac{ \pi}{ 2}}}}(\sigma_1)=\sum_{k\geq 1}\left|\operatorname{Arg}(1- ip_k^{-\sigma_1})\right|\geq \sum_{k\geq 1}\left|\operatorname{Arg}(1- ip_k^{-\sigma_2})\right|=A_{\zeta_{\boldsymbol{\frac{ \pi}{ 2}}}}(\sigma_2),$$
which proves that $A_{\zeta_{\boldsymbol{\frac{ \pi}{ 2}}}}(\sigma)$ is decreasing on $(1,\infty)$. Moreover,
by Weierstrass's criterion, the sum $$\sum_{k\geq 1}\left|\operatorname{Arg}(1- ip_k^{-\sigma})\right|=\Sum_{k\geq 1}\arctan(p_k^{-\sigma})\leq \sum_{k\geq 1}p_k^{-\sigma}$$ converges uniformly on every reduced strip of $U$, which proves the continuity of $A_{\zeta_{\boldsymbol{\frac{ \pi}{ 2}}}}(\sigma)$.
\end{proof}

As a consequence of the lemma above, given $\mathbf{x}\in \mathbb{R}^{\infty}$, the mapping $A_{\zeta_{\mathbf{x}}}(s):U\mapsto \mathbb{R}_+$ considered in \eqref{champa2} leads to a function well defined.
Furthermore, it is clear that the case $\zeta_{\boldsymbol{\frac{ \pi}{ 2}}}(s)$ is particularly significant in our context. We next prove the following lemma regarding this function.

\begin{lemma}\label{kema}
Fixed $\varepsilon>0$, the parametrized curve traced by the points  $(\operatorname{Re}(\zeta_{\boldsymbol{\frac{ \pi}{ 2}}}(\sigma)),\operatorname{Im}(\zeta_{\boldsymbol{\frac{ \pi}{ 2}}}(\sigma))$, where $\sigma$ varies in the interval $(1,1+\varepsilon)$, makes infinitely many turns around the origin.
\end{lemma}
\begin{proof}
From Lemma \ref{lemimco} (see the point iv)), we deduce that $A_{\zeta_{\boldsymbol{\frac{ \pi}{ 2}}}}(\sigma):(1,\infty)\mapsto(0,\infty)$, defined as  $A_{\zeta_{\boldsymbol{\frac{ \pi}{ 2}}}}(\sigma)=\Sum_{k\geq 1}\left|\operatorname{Arg}(1- ip_k^{-\sigma})\right|,$  is a decreasing function and it is also bijective. Likewise, by Lemma \ref{lemimco}, point iii), we have that $\Lim_{\sigma\to 1^+}A_{\zeta_{\boldsymbol{\frac{ \pi}{ 2}}}}(\sigma)=\infty$. In this way, in virtue of
$$\operatorname{Arg}(\zeta_{\boldsymbol{\frac{ \pi}{ 2}}}(\sigma))=-\sum_{k\geq 1}\operatorname{Arg}(1-ip_k^{-\sigma})=\sum_{k\geq 1}\arctan(p_k^{-\sigma}) \ (\mbox{mod} (-\pi,\pi]),$$
 it is clear that, fixed $\theta\in(-\pi,\pi]$, the equality $\operatorname{Arg}(\zeta_{\boldsymbol{\frac{ \pi}{ 2}}}(\sigma))=\theta$ is attained for infinitely many values in $(1,1+\varepsilon)$, with $\varepsilon>0$, and the
 curve originated from the points $\left(\operatorname{Re}(\zeta_{\boldsymbol{\frac{ \pi}{ 2}}}(\sigma)),\operatorname{Im}(\zeta_{\boldsymbol{\frac{ \pi}{ 2}}}(\sigma))\right)$, with $\sigma\in(1,1+\varepsilon)$, makes infinitely many turns around the origin.
\end{proof}

\begin{remark}
In view of Lemma \ref{lemcu}, it can be easily seen that the result above is also valid for other cases such as $\zeta_{\mathbf{x}}(\sigma)$ where $\mathbf{x}\in\mathbb{R}^{\infty}$ is a vector whose components are all $\pm\frac{\pi}{ 2}$ except at most a finite amount of them.
\end{remark}

Now, given $\theta\in [0,\pi]$ and $\mathbf{x}\in\mathbb{R}^{\infty}$, define $$\sigma_{\theta}^{\mathbf{x}}:=\sup\{\sigma>1:|\operatorname{Arg}(\zeta_{\mathbf{x}}(\sigma+it))|=\theta\mbox{ for some }\sigma+it\in U\}.$$
In virtue of Lemma \ref{kema} and \cite[Theorem 18]{new} it is now clear that $\sigma_{\theta}^{\mathbf{x}}$ is well defined, as the set $\{\sigma>1:|\operatorname{Arg}(\zeta_{\mathbf{x}}(\sigma+it))|=\theta\mbox{ for some }\sigma+it\in\mathbb{C}\}$ is not empty for any $\theta\in[0,\pi]$. Moreover, it
coincides with
\begin{equation}\label{venga}
\sigma_{\theta}^{\mathbf{x}}=\sup\{\sigma>1:A_{\zeta_{\mathbf{x}}}(\sigma+it))=\theta\mbox{ for some }\sigma+it\in U\}.
\end{equation}
In this respect, we next deduce from Lemma \ref{lemimco} that $\sigma_{\theta}^{\mathbf{x}}$ is bounded above by $$\sigma_{\theta}^{\boldsymbol{\frac{ \pi}{ 2}}}:=\sup\{\sigma>1:|\operatorname{Arg}(\zeta_{\boldsymbol{\frac{ \pi}{ 2}}}(\sigma+it))|=\theta\mbox{ for some }\sigma+it\in U\}$$
or
$$\sigma_{\theta}^{\boldsymbol{-\frac{ \pi}{ 2}}}:=\sup\{\sigma>1:|\operatorname{Arg}(\zeta_{\boldsymbol{-\frac{ \pi}{ 2}}}(\sigma+it))|=\theta\mbox{ for some }\sigma+it\in U\}$$
(in view of the Euler products, there is no doubt about the symmetry between $\operatorname{Arg}(\zeta_{\boldsymbol{\frac{ \pi}{ 2}}}(s))$ and $\operatorname{Arg}(\zeta_{\boldsymbol{-\frac{ \pi}{ 2}}}(s))$, which yields that $\sigma_{\theta}^{\boldsymbol{\frac{ \pi}{ 2}}}=\sigma_{\theta}^{\boldsymbol{-\frac{ \pi}{ 2}}}$).

\begin{lemma}\label{lemau}
Let $\theta\in [0,\pi]$ and $\mathbf{x}\in\mathbb{R}^{\infty}$.
Then
$
\sigma_{\theta}^{\mathbf{x}}\leq \sigma_{\theta}^{\boldsymbol{\frac{ \pi}{ 2}}}=\sigma_{\theta}^{\boldsymbol{-\frac{ \pi}{ 2}}}.
$
\end{lemma}
\begin{proof}
Given $\sigma>1$ and $\mathbf{x}\in\mathbb{R}^{\infty}$, Lemma \ref{lemimco} assures that
$$A_{\zeta_{\mathbf{x}}}(\sigma+it)\leq A_{\zeta_{\boldsymbol{\frac{ \pi}{ 2}}}}(\sigma)=A_{\zeta_{\boldsymbol{-\frac{ \pi}{ 2}}}}(\sigma).$$
Moreover, in view of \eqref{venga}, it is satisfied $$\sigma_{\theta}^{\boldsymbol{\frac{ \pi}{ 2}}}=\sup\{\sigma>1:A_{\zeta_{\boldsymbol{\frac{ \pi}{ 2}}}}(\sigma+it)=\theta\mbox{ for some }\sigma+it\in U\}.$$
Now, given $\mathbf{x}_0\in \mathbb{R}^{\infty}$, suppose by reductio ad absurdum that $\sigma_{\theta}^{\mathbf{x}_0}>\sigma_{\theta}^{\boldsymbol{\frac{ \pi}{ 2}}}$. This implies
the existence of $s_0=\sigma_0+it_0\in U$, with $\sigma_0>\sigma_{\theta}^{\boldsymbol{\frac{ \pi}{ 2}}}$, satisfying the equality
$A_{\zeta_{\mathbf{x}_0}}(\sigma_0+it_0)=\theta.$ Hence
$$\theta=A_{\zeta_{\mathbf{x}_0}}(\sigma_0+it_0)\leq A_{\zeta_{\boldsymbol{\frac{ \pi}{ 2}}}}(\sigma_0)<\theta,$$
which is a contradiction.
\end{proof}

In particular, this result yields the existence of a real number $\sigma_{\pi}>1$ such that the images of any function equivalent to the Riemann zeta function on the half-plane $\{s\in\mathbb{C}:\operatorname{Re}s>\sigma_{\pi}\}$ cannot take negative real values.

For the following result, fixed $\theta\in [0,\pi]$, take the notation $$\sigma_{\theta}:=\sup\{\sigma>1:|\operatorname{Arg}(\zeta(\sigma+it))|=\theta\mbox{ for some }\sigma+it\in U\}.$$
\begin{proposition}\label{proop8}
Let $\theta\in [0,\pi]$ and $\mathbf{x}\in\mathbb{R}^{\infty}$. Then 
\begin{itemize}
\item[i)] $|\operatorname{Arg}(\zeta_{\mathbf{x}}(s))|<\theta$ for any $s=\sigma+it\in U$ with $\sigma>\sigma_{\theta}$;\\[-0.05cm]
\item[ii)] $\sigma_{\theta}\leq\sup\{\sigma>1:\Sum_{k\geq 1}\arctan(p_k^{-\sigma})=\theta\}.$
\end{itemize}
\end{proposition}
\begin{proof}
i) By definition of $\sigma_{\theta}$, it is plain that
$$|\operatorname{Arg}(\zeta(s))|<\theta\ \forall s=\sigma+it\in U: \sigma>\sigma_{\theta}.$$
Now, the result follows from \cite[Theorem 18]{new}, as
$$\bigcup_{\sigma>\sigma_{\theta}}\operatorname{Img}\left(\zeta(\sigma+it)\right)=
\bigcup_{\sigma>\sigma_{\theta}}\operatorname{Img}\left(\zeta_{\mathbf{x}}(\sigma+it)\right).$$

ii) We already know that
$$\sigma_{\theta}=\sup\{\sigma>1:|\operatorname{Arg}(\zeta(\sigma+it))|=\theta\}=
\sup\{\sigma>1:A_{\zeta}(\sigma+it)=\theta\}.$$
By Lemma \ref{lemimco}, it is accomplished that
$$A_{\zeta_{\boldsymbol{\frac{ \pi}{ 2}}}}(s)=A_{\zeta_{\boldsymbol{-\frac{ \pi}{ 2}}}}(s)=\sum_{k\geq 1}|\arctan(p_k^{-\sigma})|=\sum_{k\geq 1}\arctan(p_k^{-\sigma}).$$
Finally, the result follows from Lemma \ref{lemau}.
\end{proof}

We next focus our attention on the vortex-like behaviour of the Riemann zeta function, and of every function that is equivalent to it. For this reason, we first show the following two preliminary results.

\begin{proposition}\label{thm}
Let $\mathbf{x}\in\mathbb{R}^{\infty}$, $\sigma_0>1$, $\theta\in(-\pi,\pi]$ and $n\in\mathbb{N}$. Then there exists a real number $\rho_n$ with $1<\rho_n<\sigma_0$ such that $\operatorname{Arg}(\zeta_{\mathbf{x}}(s))=\theta$ is satisfied for at least $n$ distinct values in the vertical strip $\{s\in \mathbb{C}:\rho_n<\operatorname{Re}s<\sigma_0\}$.
\end{proposition}
\begin{proof}
Given $\varepsilon>0$, we already know from Lemma \ref{kema} that the parametrized curve originated from the points $\left(\operatorname{Re}(\zeta_{\boldsymbol{\frac{ \pi}{ 2}}}(\sigma)),\operatorname{Im}(\zeta_{\boldsymbol{\frac{ \pi}{ 2}}}(\sigma))\right)$, with $\sigma\in(1,1+\varepsilon)$, makes infinitely many turns around the origin.
In particular, given $\sigma_0>1$, $\theta\in(-\pi,\pi]$ and $n\in\mathbb{N}$, there exists a real number $\rho_n$ with $1<\rho_n<\sigma_0$ such that $\operatorname{Arg}(\zeta_{\boldsymbol{\frac{ \pi}{ 2}}}(\sigma))=\theta$ is satisfied for at least $n$ distinct values in $(\rho_n,\sigma_0)$.
Now, given $\mathbf{x}\in\mathbb{R}^{\infty}$, we deduce from \cite[Theorem 18]{new} that
$$\bigcup_{\sigma\in (\rho_n,\sigma_0)}\operatorname{Img}\left(\zeta_{\mathbf{x}}(\sigma+it)\right)=
\bigcup_{\sigma\in (\rho_n,\sigma_0)}\operatorname{Img}\left(\zeta_{\boldsymbol{\frac{ \pi}{ 2}}}(\sigma+it)\right).$$
Then $\operatorname{Arg}(\zeta_{\mathbf{x}}(s))=\theta$ is satisfied for at least $n$ distinct values in the vertical strip $\{s\in \mathbb{C}:\rho_n<\operatorname{Re}s<\sigma_0\}$ and the result follows.
\end{proof}

\begin{proposition}\label{cor8}
Let $\mathbf{x}\in\mathbb{R}^{\infty}$, $\sigma_0>1$ and $\theta\in(-\pi,\pi]$. Then $\operatorname{Arg}(\zeta_{\mathbf{x}}(s))=\theta$ is satisfied for many infinitely values of $s\in \mathbb{C}$ with $1<\operatorname{Re}s<\sigma_0$.
\end{proposition}
\begin{proof}
Fixed $\varepsilon>0$, we will again use the fact that the curve originated from the images $\zeta_{\boldsymbol{\frac{ \pi}{ 2}}}(\sigma)$, with $\sigma\in(1,1+\varepsilon)$, makes infinitely many turns around the origin (Lemma \ref{kema}). In this way, given $\theta\in(-\pi,\pi]$,  there exists a set $C\subset (1,1+\varepsilon)$ consisting of infinitely many isolated values $\sigma$ satisfying $\operatorname{Arg}(\zeta_{\boldsymbol{\frac{ \pi}{ 2}}}(\sigma))=\theta$.
Now, given $\sigma_0\in C$, by continuity choose $\delta_{\sigma_0}>0$ such that the parametrized curve traced by the points $\left(\operatorname{Re}(\zeta_{\boldsymbol{\frac{ \pi}{ 2}}}(\sigma)),\operatorname{Im}(\zeta_{\boldsymbol{\frac{ \pi}{ 2}}}(\sigma))\right)$, with $\sigma\in(\sigma_0-\delta_{\sigma_0},\sigma_0+\delta_{\sigma_0})$, makes one turn around the origin.
Thus we deduce from \cite[Theorem 18]{new} that
$$\bigcup_{\sigma\in (\sigma_0-\delta_{\sigma_0},\sigma_0+\delta_{\sigma_0})}\operatorname{Img}\left(\zeta_{\mathbf{x}}(\sigma+it)\right)=
\bigcup_{\sigma\in (\sigma_0-\delta_{\sigma_0},\sigma_0+\delta_{\sigma_0})}\operatorname{Img}\left(\zeta_{\boldsymbol{\frac{ \pi}{ 2}}}(\sigma+it)\right).$$
Hence $\operatorname{Arg}(\zeta_{\mathbf{x}}(s))=\theta$ is satisfied at least once in the vertical strip $E_{\sigma_0}:=\{s\in\mathbb{C}:\sigma_0-\delta_{\sigma_0}<\operatorname{Re}s<\sigma_0+\delta_{\sigma_0}\}$.
Finally, by varying $\sigma_0$ in the set $C$, the result follows.
\end{proof}

Finally, by using \cite[Theorem 2]{SV}, we next improve the results above in the following sense.
\begin{theorem}\label{thm2}
Let $\mathbf{x}\in\mathbb{R}^{\infty}$, $\sigma_0>1$ and $n\in\mathbb{N}$. Then there exists a relatively dense set of real numbers $\{t_{n,m}\}_{m\geq 1}$ such that, for each $m=1,2,\ldots$, the parametrized curve traced by the points  $(\operatorname{Re}(\zeta_{\mathbf{x}}(\sigma+it_{n,m})),\operatorname{Im}(\zeta_{\mathbf{x}}(\sigma+it_{n,m})))$, where $\sigma\in (1,\sigma_0)$, makes at least $n$ turns around the origin.
\end{theorem}
\begin{proof}
It is worth noting that \cite[Corollary 5]{SV} assures the existence of an increasing unbounded sequence $\{\tau_j\}_{j\geq 1}$ of positive numbers such that the sequence of functions $\{\zeta_{\mathbf{x}}(s+i\tau_j)\}_{j\geq 1}$ converges uniformly to $\zeta_{\boldsymbol{\frac{ \pi}{ 2}}}(s)$ on every reduced strip of $U=\{s\in\mathbb{C}:\operatorname{Re}s > 1\}$. In fact, given $\sigma_0>1$, $\varepsilon>0$ and a reduced strip $U_{\varepsilon}=\{s\in\mathbb{C}:1+\varepsilon<\operatorname{Re}s<\sigma_0\}\subset U$, also by \cite[Corollary 5]{SV} there exists a relatively dense set of real numbers $\{t_j\}_{j\geq 1}$ such that
$$\left|\zeta_{\mathbf{x}}(s+it_j)-\zeta_{\boldsymbol{\frac{ \pi}{ 2}}}(s)\right|<\varepsilon\mbox{ for any }s\in U_{\varepsilon}.$$
In particular, we have
$$\left|\zeta_{\mathbf{x}}(\sigma+it_j)-\zeta_{\boldsymbol{\frac{ \pi}{ 2}}}(\sigma)\right|<\varepsilon\mbox{ for any }\sigma\in (1+\varepsilon,\sigma_0).$$
Equivalently,
\begin{equation}\label{equi}
\zeta_{\mathbf{x}}(\sigma+it_j)=\zeta_{\boldsymbol{\frac{ \pi}{ 2}}}(\sigma)+\delta_{j},\ \mbox{with }|\delta_j|<\varepsilon,\ \sigma\in (1+\varepsilon,\sigma_0).
\end{equation}
Likewise, given $\theta\in (-\pi,\pi]$, we deduce from Lemma \ref{kema} that $\operatorname{Arg}(\zeta_{\boldsymbol{\frac{ \pi}{ 2}}}(\sigma))=\theta$ is attained for a set $C$ of infinitely many values $\sigma$ in $(1,\sigma_0)$. In fact, given $n\in\mathbb{N}$, there exists $\varepsilon$ sufficiently small such that $\operatorname{Arg}(\zeta_{\boldsymbol{\frac{ \pi}{ 2}}}(\sigma))=\theta$ is satisfied for at least $n+1$ distinct values in $(1+\varepsilon,\sigma_0)$, and the curve originated from the points
$\left(\operatorname{Re}(\zeta_{\boldsymbol{\frac{ \pi}{ 2}}}(\sigma)),\operatorname{Im}(\zeta_{\boldsymbol{\frac{ \pi}{ 2}}}(\sigma))\right)$, with $\sigma\in(1+\varepsilon,\sigma_0)$, makes at least $n+1$ turns around the origin.
In terms of the decreasing function $A_{\zeta_{\boldsymbol{\frac{ \pi}{ 2}}}}(\sigma_0)$ (see Lemma \ref{lemimco}, point iv)), this means that
\begin{equation}\label{pointy}
A_{\zeta_{\boldsymbol{\frac{ \pi}{ 2}}}}(1+\varepsilon)-A_{\zeta_{\boldsymbol{\frac{ \pi}{ 2}}}}(\sigma_0)> 2\pi n.
\end{equation}
Thus, by taking $\varepsilon$ sufficiently small and taking into account that $\zeta_{\mathbf{x}}(s)\neq 0$ for any $s\in U$, we deduce from \eqref{equi} and \eqref{pointy} the existence of a sequence $\{t_{n,m}\}_{m\geq 1}$ of real numbers satisfying the following property: For each $m=1,2,\ldots$, the equality $\operatorname{Arg}(\zeta_{\mathbf{x}}(\sigma+it_{n,m}))=\theta$ is satisfied for at least $n$ values $\sigma$ in $(1+\varepsilon,\sigma_0)$ (near the points $C\cap(1+\varepsilon,\sigma_0)$) and the curve traced by the points $\left(\operatorname{Re}(\zeta_{\mathbf{x}}(\sigma+it_{n,m})),\operatorname{Im}(\zeta_{\mathbf{x}}(\sigma+it_{n,m}))\right)$, with $\sigma\in (1+\varepsilon,\sigma_0)$, makes at least $n$ turns around the origin.
Indeed, if $\sigma_1\in C\cap(1+\varepsilon,\sigma_0)$ and $\theta\in(-\pi,\pi)$ then $\operatorname{Arg}(\zeta_{\boldsymbol{\frac{ \pi}{ 2}}}(\sigma_1))=\theta$, $\operatorname{Arg}(\zeta_{\boldsymbol{\frac{ \pi}{ 2}}}(\sigma_1-\rho))>\theta$ and $\operatorname{Arg}(\zeta_{\boldsymbol{\frac{ \pi}{ 2}}}(\sigma_1+\rho))<\theta$ for values of $\rho>0$ sufficiently small in an interval $(0,a_{\sigma_1})$. Now, fixed $\varepsilon>0$, thanks to \eqref{pointy} we assure the existence of $\rho_{\varepsilon}\in (0,a_{\sigma_1})$ satisfying $\operatorname{Arg}(\zeta_{\mathbf{x}}(\sigma-\rho_{\varepsilon}+it_{n,m}))>\theta$ and $\operatorname{Arg}(\zeta_{\mathbf{x}}(\sigma+\rho_{\varepsilon}+it_{n,m}))<\theta$, which yields by continuity that $\operatorname{Arg}(\zeta_{\mathbf{x}}(\sigma+it_{n,m}))=\theta$ for some $\sigma\in(\sigma_1-\rho_{\varepsilon},\sigma_1+\rho_{\varepsilon}).$
\end{proof}

In particular, the result above can be particularized for the significant case of the Riemann zeta function.

\begin{corollary}
Let $\sigma_0>1$ and $n\in\mathbb{N}$. Then there exists a relatively dense set of real numbers $\{t_{n,m}\}_{m\geq 1}$ such that, for each $m=1,2,\ldots$, the parametrized curve traced by the points $$\left(\operatorname{Re}(\zeta(\sigma+it_{n,m})),\operatorname{Im}(\zeta(\sigma+it_{n,m}))\right),\ \mbox{with }\sigma\in (1,\sigma_0),$$ makes at least $n$ turns around the origin.
\end{corollary}

Moreover, Theorem \ref{thm2} (and other results as Lemma \ref{lemau} and Proposition \ref{proop8}) can also be immediately extended to its reciprocal sum (and all exponential sums included in its equivalence class) which is expressed as a Dirichlet series over the M\"{o}bius function $\mu(n)$ in the following terms (see \cite[p.3]{Titchmarsh}):
$$\frac{ 1}{ \zeta(s)}=\sum_{n\geq 1}\frac{\mu(n) }{ n^s}=\prod_{k=1}^{\infty}\left(1-p_k^{-s}\right),\ \operatorname{Re}s>1.$$

\bigskip

\noindent \textbf{Acknowledgements.} The first author was supported by PGC2018-097960-B-C22 (MCIU/AEI/ERDF, UE).

\bigskip

\bibliographystyle{amsplain}

\begin{thebibliography}{9}
\bibitem{Apostol} T.M. Apostol, Modular functions and Dirichlet series in number theory, Springer-Verlag, New York, 1990.

\bibitem{BohrDirichlet} H. Bohr, Z\"{u}r Theorie der allgemeinen Dirichletschen Reihen. {\it Math. Ann.}, \textbf{79} (1919), 136--156.

\bibitem{Bor} P. Borwein, S. Choi, B. Rooney and A. Weirathmueller (Eds.), The Riemann hypothesis. A resource for the afficionado and virtuoso alike, CMS books in Mathematics, Springer, New York, 2008.

\bibitem{Lehman} R.S. Lehman, On Liouville's function. {\it Math. Comput.} \textbf{14} (1960) 311--320.

\bibitem{Ineq} D.S. Mitrinovi\'{c}, J.E. Pe\v{c}ari\'{c}, A.M. Fink, Classical and new inequalities in Analysis, Kluwer Academics Publishers, Dordrecht, 1993.

\bibitem{Rigue} M. Righetti, On Bohr's equivalence theorem. {\it J. Math. Anal. Appl.}, \textbf{445} (1) (2017) 650-654; Corrigendum, \textit{ibid}. \textbf{449} (2017) 939--940. 

\bibitem{SV} J.M. Sepulcre, T. Vidal, Almost periodic functions in terms of Bohr's equivalence relation. {\it Ramanujan J.}, \textbf{46} (1) (2018) 245--267; Corrigendum, \textit{ibid}, \textbf{48} (3) (2019) 685--690.

\bibitem{SV2} J.M. Sepulcre, T. Vidal, Bohr's equivalence relation in the space of Besicovitch almost periodic functions. {\it Ramanujan J.}, \textbf{49} (3) (2019) 625--639.

\bibitem{new} J.M. Sepulcre, T. Vidal, A generalization of Bohr's equivalence theorem. {\it Complex Anal. Oper. Theory}, \textbf{13} (4) (2019) 1975--1988. 

\bibitem{Spira} R. Spira, Sets of values of general Dirichlet series. {\it Duke Math. J.} \textbf{35} (1) (1968) 79--82.

\bibitem{Titchmarsh} E.C. Titchmarsh, The theory of the Riemann zeta-function, Oxford Science publication,
second edition, 1986.

\end{thebibliography}

\bigskip

%
%
%
%
%
%
%
%
%
%
%
%
%
%
%
%
%

\end{document}